\documentclass[11pt, twoside]{article}
\usepackage{amssymb, amsmath, amsthm}
\usepackage[top=30mm, left=25mm, right=25mm, bottom=22mm]{geometry}
\usepackage{fancyhdr}
\usepackage{tikz}
\usetikzlibrary{positioning}
\usepackage{titling}
\usepackage{hyperref}
\predate{}
\postdate{}
\usepackage{lipsum}
\newtheorem{theorem}{Theorem}
\newtheorem{lemma}[theorem]{Lemma}

\newtheorem{claim}{Claim}

\newtheorem{conjecture}[theorem]{Conjecture}
\usepackage{titling}
\usepackage{xcolor}
\usepackage{placeins}
\usepackage{tikz}
\usepackage{verbatim}
\usepackage{titling}
\usepackage[T1]{fontenc}
\usepackage{currvita}
\usepackage{bbm}
\usepackage{enumitem}

\theoremstyle{definition}

\theoremstyle{remark}

\newtheorem*{theorem*}{Theorem}
\usepackage[T1]{fontenc}
\usepackage[utf8]{inputenc}
\usepackage{lmodern}
\usepackage{listings}
\usepackage[capitalise,noabbrev]{cleveref}
\begin{document}
\newcommand{\Addresses}{
\bigskip
\footnotesize

\medskip

\noindent Maria-Romina~Ivan, \textsc{Department of Pure Mathematics and Mathematical Statistics, Centre for Mathematical Sciences, Wilberforce Road, Cambridge, CB3 0WB, UK.} \par\noindent\nopagebreak\textit{Email address: }\texttt{mri25@dpmms.cam.ac.uk}

\medskip

\noindent Bernardus~A.~Wessels, \textsc{School of Mathematics, Georgia Institute of Technology, Atlanta, GA 30332, USA.}\par\noindent\nopagebreak\textit{Email address: }\texttt{bwessels3@gatech.edu}
\normalsize}
\pagestyle{fancy}
\fancyhf{}
\fancyhead [LE, RO] {\thepage}
\fancyhead [CE] {MARIA-ROMINA IVAN AND BERNARDUS A. WESSELS}
\fancyhead [CO] {A NEW LOWER BOUND FOR THE DIAGONAL POSET RAMSEY NUMBERS}
\renewcommand{\headrulewidth}{0pt}
\renewcommand{\l}{\rule{6em}{1pt}\ }
\title{\Large{\textbf{A NEW LOWER BOUND FOR THE DIAGONAL POSET RAMSEY NUMBERS}}}
\author{MARIA-ROMINA IVAN AND BERNARDUS A. WESSELS}
\date{ }
\maketitle
\begin{abstract}
Given two finite posets $\mathcal P$ and $\mathcal Q$, their Ramsey number, denoted by $R(\mathcal P,\mathcal Q)$, is defined to be the smallest integer $N$ such that any blue/red colouring of the vertices of the hypercube $Q_N$ has either a blue induced copy of $\mathcal P$, or a red induced copy of $\mathcal Q$.

Axenovich and Walzer showed that, for fixed $\mathcal P$, $R(\mathcal P, Q_n)$ grows linearly with $n$. However, for the diagonal question, we do not even come close to knowing the order of growth of $R(Q_n,Q_n)$. The current upper bound is $R(Q_n,Q_n)\leq n^2-(1-o(1))n\log n$, due to Axenovich and Winter.

What about lower bounds? It is trivial to see that $2n\leq R(Q_n,Q_n)$, but surprisingly, even an incremental improvement required significant work. Recently, an elegant probabilistic argument of Winter gave that, for large enough $n$, $R(Q_n,Q_n)\geq 2.02n$.

In this paper we show that $R(Q_n,Q_n)\geq 2.7n+k$, where $k$ is a constant. Our current techniques might in principle show that in fact, for every $\epsilon>0$, for large enough $n$, $R(Q_n,Q_n)\geq (3-\epsilon)n$. Our methods exploit careful modifications of layered-colourings, for a large number of layers. These modifications are stronger than previous arguments as they are more constructive, rather than purely probabilistic. 
\end{abstract}

\section{Introduction}

Given two finite posets, $\mathcal P$ and $\mathcal Q$, we denote by $R(\mathcal P,\mathcal Q)$ the minimum integer $N$ such that any blue/red copy of the hypercube $Q_N$ contains either a blue induced copy of $\mathcal P$, or a red induced copy of $\mathcal Q$. $R(\mathcal P,\mathcal Q)$ is known as the \textit{Ramsey number of posets $\mathcal P$ and $\mathcal Q$}.

Since any finite poset can be embedded in a hypercube, one of the most natural questions is: what is $R(Q_m, Q_n)$? When $m$ is fixed and $n$ grows, these are known as the \textit{off-diagonal} Ramsey numbers for posets, and $R(Q_n,Q_n)$ is known as the \textit{diagonal} Ramsey numbers for posets. These notions were first studied by Axenovich and Walzer in 2017 \cite{axenovichBooleanLatticesRamsey2017}. They also established the first bounds, namely that for any positive integers $m,n$ we have  $$n+m\leq R(Q_m,Q_n)\leq mn+m+n.$$
This already shows that the off-diagonal Ramsey numbers have linear growth. However, it places $R(Q_n, Q_n)$ between $2n$ and $n^2+2n$, and there is much desire to close the gap between these bounds.

The most recent upper bound is due to Axenovich and Winter \cite{log} who showed that $$R(Q_n,Q_n)\leq n^2-\left(1-\frac{2}{\sqrt{\log n}}\right)n\log n.$$
On the other hand, the lower bound is notoriously more elusive. While $2n$ is trivially achieved by colouring the sets of $Q_{2n-1}$ blue if the size is at most $n-1$, and red otherwise (there is no monochromatic chain on length $n+1$, let alone an entire $Q_n$), any little improvement took time, clever optics and new techniques.

First, Cox and Stolee improved the upper bound to $2n+1$ for $n\geq 13$ and $3\leq n\leq8$ \cite{cox}. Their colouring still was blue up to a certain level, and red after a certain level, except that in-between they managed to add one extra level using probabilistic methods. Their bound was later achieved by Bohman and Peng for all $n\geq3$, this time by an explicit colouring \cite{constructible}. Moreover, Gr\'osz, Methuku and Thompkins showed that, in fact, for any positive integers $n,m$, $R(Q_m,Q_n)\geq m+n+1$, which is surprisingly still the best lower bound for the off-diagonal case \cite{off-diagonal}.

Recently, using clever probabilistic methods, Winter \cite{winter} achieved the best lower bound yet for diagonal Ramsey numbers for posets, namely that for large enough $n$, $$R(Q_n,Q_n)\geq2.02n.$$
In this paper we show the following.
\begin{theorem}
There exists a constant $k$ such that $R(Q_n,Q_n)>2.7n+k$.
\end{theorem}
Our blueprint for this result can, in principle, give a lower bound of $(3-\epsilon)n$ for every fixed $\epsilon>0$ and $n$ large enough. We explore the absolute limitations of this method at the end of the paper.

Before explaining the strategy, we introduce some standard terminology. For a natural number $n$, we define $[n] = \{1, 2, \dots, n\}$. The \textit{$n$-dimensional hypercube}, denoted  by $Q_n$, is the poset consisting of all subsets of $[n]$ ordered by inclusion. An embedding of a poset $\mathcal{P}$ into a poset $\mathcal{Q}$ is an injective map $\phi: \mathcal{P} \to \mathcal{Q}$ that preserves the poset structure -- that is, $x \le_{\mathcal{P}} y$ if and only if $\phi(x) \le_{\mathcal{Q}} \phi(y)$. For $0\leq i\leq N$, we refer to the collection of subsets of $[N]$ of size $i$, denoted by $[N]^{(i)}$, as the \textit{$i^\text{th}$ level} of $Q_N$. A colouring of $Q_N$ is called \textit{layered} if sets of the same size have the same colour. The number of layers is the number of colour transitions, plus one, when going from the empty set to the full set. A \textit{layer} is a maximal collection of monochromatic levels of $Q_N$.

Our proof is inspired by Winter's proof in which a $4$-layered colouring is modified with the help of two families of sets, called \textit{pivots}. Our strategy is to start with a large number of layers, pair them up, and inside each pair modify the colouring according to two families of pivots tailored for said pair. The pivots (and their parameters) are designed to force a monochromatic embedding to `skip' a certain number of levels. If one manages to force a monochromatic embedding to skip enough levels in total, there is, of course, not enough levels left to embed a monochromatic $Q_n$.

In Winter's proof, the existence of pivots relies exclusively on probabilistic methods. Our proof is more constructive in the sense that we pinpoint exactly one type of families of pivots. This gives more control and consequently tighter bounds. It turns out that starting with 602 layers is enough to achieve the claimed $2.7n$ bound. Moreover, for an arbitrary number of initial layers, we pinpoint the exact restrictions that need to be satisfied in order for the proof to go through, which suggests that the more layers, the better the bound, although not exceeding $3n$.

The plan of the paper is as follows. In Section 2, to illustrate our methodology without overly complex parameters, we establish a weaker lower bound of $(2+1/3)n$, assuming the existence of pivot sets, for a 6-layered initial colouring. In Section 3, the most technical part of the paper, we rigorously prove the existence of the necessary pivot sets (Lemma~\ref{lem:pivot_points} and Lemma~\ref{lem:above_pivots}). Finally, in Section 4, we present the proof of our main result, for which the parameters were found using a numerical optimiser (see Appendix), as well as establish the absolute technical limitations of this type of proof strategy.
\section{A lower bound of $(2+\frac{1}{3})n$}
In this section, we show the following.
\begin{theorem}\label{thm:qnqn_lower}
For sufficiently large $n$, $R(Q_n,Q_n)>\frac{7}{3}n-\frac{413}{3}$.
\end{theorem}
The proof is heavily based on the existence of certain sets of `pivots'. We first state the lemmas that ensure their existence and use them to achieve the desired lower bound. In the next section, we provide full proofs of these lemmas. We also use the following result about embeddings of $Q_n$ (Theorem 9 in \cite{axenovichBooleanLatticesRamsey2017}). Since the proof is short, we include it here too.

\begin{lemma}[Axenovich-Walzer \cite{axenovichBooleanLatticesRamsey2017}]\label{lem:embedding}
Let $Z$ be a set such that $|Z|>n$, and let  $Q=\mathcal Q(Z)$, the hypercube with ground set $Z$. If there exists an embedding $\phi:Q_n\to Q$, then there exist a subset $X\subset Z$ such that $|X|=n$, and an embedding $\phi':\mathcal{Q}(X)\to Q$, with the same image as $\mathcal\phi$, such that $\phi'(S)\cap X=S$ for all $S\subseteq X$.
\end{lemma}
\begin{proof}
We start by looking at the image of the singletons under the original embedding, namely $\phi(\{i\})$ for all $i\in[n]$. Assume that there exist $i\in[n]$ such that $$\phi(\{i\})\subseteq\bigcup_{A\subseteq([n]\setminus\{i\})}\phi(A).$$ Since $\phi$ is an embedding, we then get $\phi(\{i\})\subseteq\phi([n]\setminus\{i\})$, a contradiction. Therefore, for every $i\in[n]$ there exists $a(i)\in\phi(\{i\})\setminus\bigcup_{A\subseteq([n]\setminus\{i\})}\phi(A)$. We further notice that if $i\neq j$, $\phi(\{j\})$ and $\phi(\{i\})\setminus\bigcup_{A\subseteq([n]\setminus\{i\})}\phi(A)$ are disjoint as $\{j\}\subseteq[n]\setminus\{i\}$. As such, $a(i)\neq a(j)$ if $i\neq j$. Let $X=\{a(i):i\in[n]\}$, which is in bijection with $[n]$. Given $A\subseteq X$, let $\hat{A}=\{a^{-1}(x):x\in A\}$. We therefore define $\phi':\mathcal Q(X)\to Q$ as follows $$\phi'(A)=\phi(\hat{A}).$$
It is clear that $\phi'$ is an embedding with the same image at $\phi$. Moreover, let $x=a(i)\in X$ for some $i\in[n]$, and $A\subseteq X$. If $x=a(i)\in A$, then $i\in\hat{A}$, which implies that $\phi(\{i\})\subseteq\phi(\hat{A})$, thus $x=a(i)\in\phi(\hat{A})=\phi'(A)$. Conversely, if $x=a(i)\in\phi'(A)=\phi(\hat{A})$, which by construction of $a$ implies that $i\in\hat{A}$, thus $a(i)=x\in A$. Therefore, for any $A\subseteq X$ we have that $\phi'(A)\cap X=A.$
\end{proof}

Assume that $n$ is a multiple of 60, and let $c=\frac{1}{3}$ and $N=(2+c)n$. In order to achieve the claimed results, it is enough to provide a colouring of $Q_N$ without a monochromatic copy of $Q_n$, for every $n$ multiple of $60$. This is because if $ 60k\leq n'\leq60k+59$, then $R(Q_{n'}, Q_{n'})\geq R(Q_{60k}, Q_{60k})>(2+c)60k\geq (2+c)(n'-59)=(2+c)n'-59(2+c).$ 

Now let $s=\frac{2n}{3}$ and $t=\frac{2n}{3}+\frac{cn}{2}+hn$, where $h=0.05$. Notice that since $n$ is assumed to be a multiple of 60, all these quantities are natural numbers.
\begin{lemma}\label{lem:pivot_points}
For large enough $n$, divisible by 60, there exist sets $\mathcal S_1\subseteq[N]^{(s)}$ and $\mathcal T_1\subseteq [N]^{(t)}$ with no $S\in\mathcal S_1$ and $T\in\mathcal T_1$ such that $S\subseteq T$, and the following properties:
\begin{enumerate}
\item For any $P,X\subseteq[N]$ with $|P|=\frac{n}{3}$, $|X|=n$ and $P\cap X=\emptyset$, there exists $S\in\mathcal S_1$ such that:
\begin{enumerate}
\item $P\cap X\subseteq S\cap X$
\item $S\setminus X\subseteq P\setminus X$
\item $|{S\cap X}|\leq\frac{n}{3}$.
\end{enumerate}
\item For any $P,X\subseteq[N]$ with $|P|=\frac{N}{2}=n+\frac{cn}{2}$, $|X|=n$, and $|{P\cap X}|= \frac{2n}{3}$, there exists some $T\in \mathcal T_1$ such that:
\begin{enumerate}
\item $T\cap X\subseteq P\cap X$
\item $P\setminus X\subseteq T\setminus X$
\item $|{T\cap X}|\ge \frac{n}{3}$.
\end{enumerate}
\end{enumerate}
\end{lemma}

By taking complements, we also get the following.
\begin{lemma}\label{lem:above_pivots}
For large enough $n$, there exist sets $\mathcal S_2\subseteq[N]^{(N-t)}$ and $\mathcal T_2\subseteq [N]^{(N-s)}$ with no $S\in\mathcal S_2$ and $T\in\mathcal T_2$ such that $S\subseteq T$, and the following properties:
\begin{enumerate}
\item For any $P,X\subseteq[N]$ with $|P|=\frac{N}{2}=n+\frac{cn}{2}$ and $|X|=n$, and $|P\cap X|= \frac{n}{3}$, there exists $S\in \mathcal S_2$ such that:
\begin{enumerate}
\item $P\cap X\subseteq S\cap X$
\item $S\setminus X\subseteq P\setminus X$
\item $|S\cap X|\leq \frac{2n}{3}$.
\end{enumerate}
\item For any $P,X\subseteq[N]$ with $|P|=N-\frac{n}{3}$, $|X|=n$, and $X\subseteq P$, there exists $T\in\mathcal T_2$ such that:
\begin{enumerate}
\item $T\cap X\subseteq P\cap X$
\item $P\setminus X\subseteq T\setminus X$
\item $|T\cap X|\geq \frac{2n}{3}$.
\end{enumerate}
\end{enumerate}
\end{lemma}

These two lemmas are enough to now construct a red-blue colouring of $Q_N$ without a monochromatic copy of $Q_n$.

\begin{proof}[Proof of Theorem~\ref{thm:qnqn_lower}]
Recall that $N=(2+\frac{1}{3})n$, $c=\frac{1}{3}$, and $n$ is a multiple of 60, large enough such that Lemma~\ref{lem:pivot_points} and Lemma~\ref{lem:above_pivots} apply. As mentioned above, all fractions below are in fact positive integers.

Therefore, by Lemma~\ref{lem:pivot_points} and Lemma~\ref{lem:above_pivots}, we have 4 families of elements, $\mathcal S_1,\mathcal T_1,\mathcal S_2, \mathcal T_2$ lying on the $s$-level, $t$-level, $(N-s)$-level and $(N-t)$-level of $Q_N$, respectively. We will call a set in any of these families a \textit{pivot}. We are going to use their existence to modify a layered colouring of $Q_N$ depending on certain up-sets and down-sets corresponding to these pivots. 

More precisely, we now colour $Q_N$ as follows. Let $A\in Q_N$.
\begin{itemize}
\item If $|A|<\frac{n}{3}$, we colour $A$ {\color{blue}{blue}}.
\item If $\frac{n}{3}\leq |A|\le \frac{2n}{3}+\frac{cn}{4}$, we colour $A$ {\color{blue}blue} if there exists $S\in\mathcal S_1$ such that $S\subseteq A$, and {\color{red}red} otherwise.
\item If $\frac{2n}{3}+\frac{cn}{4}<|A|\leq N/2$, we colour $A$ {\color{red}{red}} if there exists $T\in\mathcal T_1$ such that $A\subseteq T$, and {\color{blue}blue} otherwise.
\item If $N/2<|A|\leq N-\frac{2n}{3}-\frac{cn}{4}$, we colour $A$ {\color{blue}blue} if there exists $S\in\mathcal S_2$ such that $S\subseteq A$, and {\color{red}red} otherwise.
\item If $N-\frac{2n}{3}-\frac{cn}{4}< |A|\leq N-\frac{n}{3}$, we colour $A$ {\color{red}red} if there exists $T\in\mathcal T_2$ such that $A\subseteq T$, and {\color{blue}blue} otherwise.
\item If $N-\frac{n}{3}<|A|$, we colour $A$ {\color{red}red}.
\end{itemize}
This colouring is illustrated in the picture below.
\begin{center}
\includegraphics[width=10.5cm]{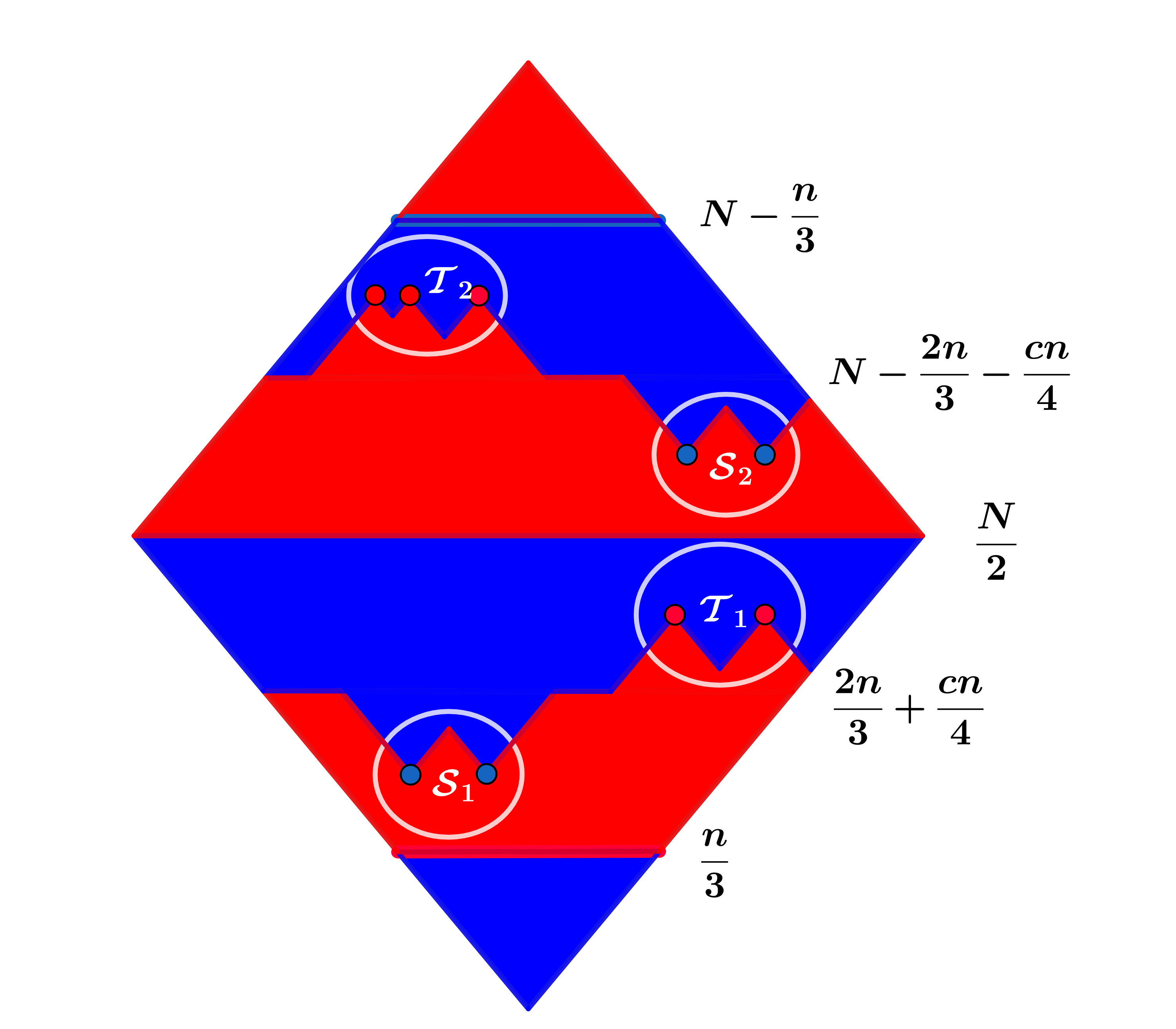}
\end{center}
Suppose now that there exists an embedding $\phi:Q_n\to Q_N$ such that its image is red. By Lemma~\ref{lem:embedding}, there exists $X\subseteq [N]$ and an embedding $\phi':\mathcal Q({X})\to Q_N$ such that $|X|=n$, $\phi'(S)\cap X=S$ for all $S\subseteq X$, and the image of $\phi$ is precisely the image of $\phi'$. Therefore $\phi'(S)$ is red for all $S\subseteq X$.

In particular, $\phi'(\emptyset)$ is red. Since in our colouring all sets of size less than $\frac{n}{3}$ are blue, we must have $|{\phi'(\emptyset)}|\ge\frac{n}{3}$. Let $P_1\subseteq\phi'(\emptyset)$ such that $|{P_1}|=\frac{n}{3}$. We notice that $P_1\cap X\subseteq \phi'(\emptyset)\cap X=\emptyset$. Therefore, by Lemma~\ref{lem:pivot_points}, part 1, there exists a pivot $S_1\in\mathcal S_1$ such that:
\begin{enumerate}
\item $P_1\cap X\subseteq S_1\cap X$
\item $S_1\setminus X\subseteq P_1\setminus X$
\item $|{S_1\cap X}|\le \frac{n}{3}$.
\end{enumerate}
We now have that $S_1=(S_1\cap X)\cup(S_1\setminus X)$. By the embedding property we have that $S_1\cap X=\phi'(S_1\cap X)\cap X$. Combining this with the second property of being a pivot, we have that $$S_1\subseteq (\phi'(S_1\cap X)\cap X)\cup (P_1\setminus X).$$
Next, since $P_1\subseteq\phi'(\emptyset)$, we of course have that $(\phi'(S_1\cap X)\cap X)\cup (P_1\setminus X)\subseteq (\phi'(S_1\cap X)\cap X)\cup (\phi'(\emptyset)\setminus X)$. Moreover, since $\phi'$ is a hypercube embedding, we have that $\phi'(\emptyset)\subseteq\phi'(S_1\cap X)$. Putting everything together, we get $$S_1\subseteq (\phi'(S_1\cap X)\cap X)\cup (\phi'(S_1\cap X)\setminus X)=\phi'(S_1\cap X).$$

In particular, this implies that $S_1\cap X\neq\emptyset$. Since $S_1\in\mathcal S_1$ is a pivot, we coloured it and everything above it of size at most $\frac{2n}{3}+\frac{cn}{4}$ blue. Moreover, since no element of $\mathcal S_1$ is a subset of any element of $\mathcal T_1$, we also have that all sets above $S_1$ of size less than $\frac{N}{2}$ are blue. Therefore we must have that $|\phi'(S_1\cap X)|\geq\frac{N}{2}$. As before, let $P_2\subseteq \phi'(S_1\cap X)$ such that $|{P_2}|=\frac{N}{2}$. We now see that $P_2\cap X\subseteq \phi'(S_1\cap X)\cap X=S_1\cap X$, thus $|P_2\cap X|\leq |S_1\cap X|\leq\frac{n}{3}$.
We want $P_2\cap X$ to have exactly $\frac{n}{3}$ elements, so define $P_2'$ by adding arbitrary elements from $X\setminus P_2$ to $P_2$ such that $|P_2'\cap X|=\frac{n}{3}$, and removing elements from $P_2\setminus X$ so that $|P_2'|=\frac{N}{2}$.
Therefore, by Lemma~\ref{lem:above_pivots}, part 1, there exists a pivot $S_2\in\mathcal S_2$ such that:
\begin{enumerate}
\item $P_2\cap X\subseteq P_2'\cap X\subseteq S_2\cap X$
\item $S_2\setminus X\subseteq P_2'\setminus X\subseteq P_2\setminus X$
\item $|{S_2\cap X}|\leq \frac{2n}{3}$.
\end{enumerate}
Following the same blueprint as above, we similarly get that
$$S_2\subseteq \phi'(S_2\cap X).$$
As before, since $S_2\in\mathcal S_2$ is a pivot, everything above it of size at least $N-\frac{n}{3}$ is blue. Therefore, as $\phi'(S_2\cap X)$ is red, we must have that $|\phi'(S_2\cap X)|>N-\frac{n}{3}$. 

On the other hand, we have that $|S_2\cap X|\le \frac{2n}{3}$, which implies that $|X\setminus S_2|\geq\frac{n}{3}$. Finally, since $\phi'(S_2\cap X)\cap (X\setminus S_2)=\emptyset$, we have that $|\phi'(S_2\cap X)|\leq N-|X\setminus S_2|\le N-\frac{n}{3}$, a contradiction. Therefore there exists no red embedding.
    
By the symmetry of the colouring and of the conditions of Lemma~\ref{lem:pivot_points} and Lemma~\ref{lem:above_pivots} (under taking complements), there also does not exist a blue embedding. The argument will start with looking at the top of the embedding, and taking a superset of it of size $N-\frac{n}{3}$. We then go down the layers, using the families of pivots $\mathcal T_1$ and $\mathcal T_2$, in the same way we went up the layers for the red embedding, using the families of pivots $\mathcal S_1$ and $\mathcal S_2$. Eventually we will find some $A\in\ Q_N$ such that $\frac{n}{3}\leq |A|<\frac{n}{3}$, which gives the desired contradiction, and finishes the proof.
\end{proof}
\section{Existence of pivot sets}
In this section we prove in full Lemma~\ref{lem:pivot_points} and Lemma~\ref{lem:above_pivots}. We first focus on Lemma~\ref{lem:pivot_points} and then, by taking complements, we establish Lemma~\ref{lem:above_pivots}. To refresh the reader's memory, we restate them.
\setcounter{theorem}{3}
\begin{lemma}
For large enough $n$, divisible by 60, there exist sets $\mathcal S_1\subseteq[N]^{(s)}$ and $\mathcal T_1\subseteq [N]^{(t)}$ with no $S\in\mathcal S_1$ and $T\in\mathcal T_1$ such that $S\subseteq T$, and the following properties:
\begin{enumerate}
\item For any $P,X\subseteq[N]$ with $|P|=\frac{n}{3}$, $|X|=n$ and $P\cap X=\emptyset$, there exists $S\in\mathcal S_1$ such that:
\begin{enumerate}
\item $P\cap X\subseteq S\cap X$
\item $S\setminus X\subseteq P\setminus X$
\item $|{S\cap X}|\leq\frac{n}{3}$.
\end{enumerate}
\item For any $P,X\subseteq[N]$ with $|P|=\frac{N}{2}=n+\frac{cn}{2}$, $|X|=n$, and $|{P\cap X}|= \frac{2n}{3}$, there exists some $T\in \mathcal T_1$ such that:
\begin{enumerate}
\item $T\cap X\subseteq P\cap X$
\item $P\setminus X\subseteq T\setminus X$
\item $|{T\cap X}|\ge \frac{n}{3}$.
\end{enumerate}
\end{enumerate}
\end{lemma}
\begin{proof}[Proof of Lemma~\ref{lem:pivot_points}]
We recall that $c=\frac{1}{3}$, $h=0.05n$, $s=\frac{2n}{3}$, and $t=\frac{2n}{3}+\frac{cn}{2}+hn$, and since $n$ is divisible by $60$, $s, t, \frac{cn}{2},hn\in\mathbb N$. Let us start with $\mathcal S'=[N]^{(s)}$. We take $\mathcal T_1$ to be a family constructed by adding each set of size $t$, independently, with probability $p=(0.525)^n$. We will now modify $\mathcal S'$ so that together with $\mathcal T_1$, all conditions in the statement are satisfied. In order to do this, for any pair of sets that satisfies either condition of the lemma, we look at the family of all candidate sets from $\mathcal S'$. We call these families \textit{cones}.

More precisely, given $P,X\subseteq[N]$ such that $|{P}|=\frac{n}{3}$, $|{X}|=n$, and $P\cap X=\emptyset$, we define the \textit{$s$-cone of $P$ and $X$} as follows:
$$\mathcal K_s(P,X)=\bigg\{S\subseteq[N]: |{S}|=s,\,S\setminus X\subseteq P\setminus X,\,|{S\cap X}|\le\frac{n}{3}\bigg\}.$$
Similarly, for any $P,X\subseteq [N]$ such that $|{P}|=\frac{N}{2}=n+\frac{cn}{2}$, $|{X}|=n$, and $|{P\cap X}|= \frac{2n}{3}$, we define the \textit{$t$-cone of $P$ and $X$} as follows:
$$\mathcal K_t(P,X)=\bigg\{T\subseteq[N]:|{T}|=t,\,T\cap X\subseteq P\cap X,\,P\setminus X\subseteq T\setminus X,\,|{T\cap X}|\ge\frac{n}{3}\bigg\}.$$

For readability purposes, from now on we will write $\mathcal K_s$ to mean a cone $\mathcal K_s(P,X)$ for some sets $P$ and $X$ such that $|{P}|=\frac{n}{3}$, $|{X}|=n$, and $P\cap X=\emptyset$. Similarly, we adopt the same shorthand notation for $t$-cones, namely $\mathcal K_t$.

The proof has three building blocks. First, we show that all $s$-cones $\mathcal K_s$ are large, and that with high probability, all intersections $\mathcal K_t\cap\mathcal T$ are large.
        
Next, we want to focus on the property that no element of $\mathcal S_1$ can be a subset of any element of $\mathcal T_1$. Therefore we say that cone $\mathcal K_s$ is \textit{bad} if for all $S\in\mathcal K_s$ there is some $T\in\mathcal T_1$ such that $S\subseteq T$. We then show that for any cone, $\mathbb{P}[\mathcal K_s\text{ is bad}]\leq (0.9997)^{n(1.0002)^n}$.

Finally, we find that with high probability there are no bad cones, which allows us to build legal families $\mathcal S_1$ and $\mathcal T_1$, as claimed.

Before delving into the proof, we establish a useful estimate for the choose function which we will use throughout. We start with the following inequalities. Let $n$ be a positive integer. Then 
$$\sqrt{2\pi}n^{n+0.5}e^{-n}\leq n!\leq en^{n+0.5}e^{-n}.$$ This can be proven relatively quickly by analysing the sequence $a_n=\frac{n!}{e^{-n}n^{n+0.5}}$. By looking at $\log(a_n/a_{n+1})$, one easily finds that the sequence is decreasing. Since $a_1=e$ and, by Stirling's formula, $\lim_{n\to\infty}a_n=\sqrt{2\pi}$, we indeed get $\sqrt{2\pi}\leq \frac{n!}{e^{-n}n^{n+0.5}}\leq e$, as claimed.

Next, using the above inequalities, we have that, given positive constants $d<C$ and a natural number $n$ such that $dn, Cn\in\mathbb N$,

$$O(1) \frac{1}{\sqrt n}\left(\frac{C^C}{d^d(C-d)^{C-d}}\right)^n\leq{Cn\choose dn}=\frac{(Cn)!}{(dn)!(Cn-dn)!}\leq O(1) \frac{1}{\sqrt n}\left(\frac{C^C}{d^d(C-d)^{C-d}}\right)^n.$$
    
In particular, this means that for any $\epsilon>0$, and large enough $n$, we have that:
$$\left(\frac{C^C}{d^d(C-d)^{C-d}}-\epsilon\right)^n<{Cn\choose dn}< \left(\frac{C^C}{d^d(C-d)^{C-d}}+\epsilon\right)^n.$$

In what follows we will assume that $n$ is large enough, thus apply the above inequalities by rounding the numerical values either up or down, depending on whether we want an upper or lower bound.

\begin{claim}\label{clm_1}
For every cone $\mathcal K_s$, $|{\mathcal K_s}|> (1.8898)^n$, and with high probability $|{\mathcal K_t\cap\mathcal T}|>(1.006)^n$ for every cone $\mathcal K_t$.
\end{claim}
\begin{proof}
Let $\mathcal K_s$ be an $s$-cone with the associated sets $P,X$. We want to count the size of the cone. We are looking for sets $S$ of size $\frac{2n}{3}$ such that $S\setminus X\subseteq P\setminus X$ and $|S\cap X|\leq\frac{n}{3}$. Therefore selecting exactly $\frac{n}{3}$ elements of $X$, and taking the union with $P$ gives us a set in the cone. More precisely
$$\bigg\{S:S=P\cup S',\ S'\subset X,\ |S'|=\frac{n}{3}\bigg\}\subseteq\mathcal K_s(P,X).$$
Therefore we have that $|\mathcal K_s|\geq {n\choose n/3}$. By the above inequality, applied with $d=\frac{1}{3}, C=1$, we get that for any $\epsilon>0$, for large enough $n$, $|\mathcal K_s|\geq \left(\frac{1}{(1/3)^{1/3}(2/3)^{2/3}}-\epsilon\right)^n$. Since $\frac{1}{(1/3)^{1/3}(2/3)^{2/3}}=\frac{3}{\sqrt[3]{4}}=1.88988\dots>1.8898$, we indeed get that for large enough $n$, all $s$-cones have size at least $(1.8898)^n$.

\vspace{0.5em}
Next, let $\mathcal K_t$ be a $t$-cone with associated sets $R$ and $Y$. We recall that this means that $|R|=\frac{N}{2}=n+\frac{cn}{2}$, $|Y|=n$, and $|R\cap Y|=\frac{2n}{3}$. Also $t=\frac{2n}{3}+\frac{cn}{2}+hn$.

Again, based on these properties, we will pinpoint a large number of sets that must be in the cone. Indeed, consider sets of size $t$ formed by taking the union of $R\setminus Y$ (which has size $\frac{n}{3}+\frac{cn}{2}$) with $\frac{n}{3}$ elements from $R\cap Y$ (which is possible as $|R\cap Y|=\frac{2n}{3}$), plus at least $hn$ other elements from $[N]\setminus(R\cup Y)$. 

To ensure that this is possible, we must have that $|[N]\setminus(R\cup Y)|\geq t-\frac{n}{3}-|R\setminus Y|$. This is true as $|{[N]\setminus(R\cup Y)}|=N-|{R}|-|{Y}|+|Y\cap R|= (2+c)n-\left(n+\frac{cn}{2}\right)-n+\frac{2n}{3}=\frac{2n}{3}+\frac{cn}{2}=t-hn$. Clearly $t-hn$ is greater than $t-\frac{n}{3}-|R\setminus Y|$. Therefore we have that
$$\left\{T=(R\setminus Y)\cup S'\cup S'':|T|=t,\ |S'|=\frac{n}{3},\ S'\subset R\cap Y,\ S''\subseteq [N]\setminus(R\cup Y)\right\}\subseteq \mathcal K_t(R,Y).$$
Based on this we have a lower bound for $\mathcal K_t$, namely that
$$|{\mathcal K_t}|\geq {2n/3\choose n/3}{2n/3 + cn/2\choose hn}.$$
Therefore, applying our inequality for the choose function, we get that for any $\epsilon>0$, and $n$ large enough
$$|\mathcal K_t|\geq\left(\frac{(2/3)^{2/3}}{(1/3)^{1/3}(1/3)^{1/3}}-\epsilon\right)^n\left(\frac{(2/3+1/6)^{2/3+1/6}}{0.05^{0.05}(2/3+1/6-0.05)^{2/3+1/6-0.05}}-\epsilon\right)^n.$$
Since $\frac{(2/3)^{2/3}}{(1/3)^{1/3}(1/3)^{1/3}}\frac{(2/3+1/6)^{2/3+1/6}}{0.05^{0.05}(2/3+1/6-0.05)^{2/3+1/6-0.05}}=1.917913\dots>1.9179$, we have that for large enough $n$ any $t$-cone has size at least $(1.9179)^n$.

This implies that, for large enough $n$, we have 
$$\mathbb{E}[|{\mathcal K_t\cap\mathcal T|}]=|{\mathcal K_t}|p> (1.9179^n)(0.525^n)> (1.0068)^n>2(1.006^n).$$

To finish the claim, we are left to show that the expected number of $t$-cones whose intersection with $\mathcal T$ is at most $(1.006)^n$ goes to zero as $n$ goes to infinity. To do this, we will make use of the multiplicative Chernoff's inequality which states that for a binomial random variable $X$ and a real number $0<a<1$, the following is true:
 $$\mathbb{P}[X\le (1-a)\mathbb{E}[X]]\leq\exp\left(\frac{-\mathbb{E}[X]a^2}{2}\right).$$
We take $a$ to be $\frac{1}{2}$ and the random variable $|{\mathcal K_t\cap\mathcal T}|$. Thus, we get:
\begin{align*}
\mathbb{P}[|{\mathcal K_t\cap\mathcal T|}\leq(1.006)^n]&\leq\mathbb{P}\left[|{\mathcal K_t\cap\mathcal T|}\leq\left(1-\frac{1}{2}\right)\mathbb{E}[|{\mathcal K_t\cap\mathcal T|}]\right]\\
&\le \exp\left(\frac{-\mathbb{E}[|{\mathcal K_t\cap\mathcal T|}]}{8}\right)\leq \exp\left(\frac{-(1.006)^n}{4}\right).
\end{align*}
Denote by $X_t$ the random variable equal to the number of $t$-cones $\mathcal K_t$ such that $|{\mathcal K_t\cap\mathcal T}|\leq (1.006)^n$. We therefore have that
\begin{align*}
\mathbb{E}[X_t]&=\sum_{\mathcal K_t}\mathbb{P}[|{\mathcal K_t\cap\mathcal T}|\leq (1.006)^n]\\
&\leq \sum_{\substack{X\subseteq[N],\\|{X}|=n}}\sum_{\substack{P\subseteq[N],\\|{P}|=n+cn/2,\\|{P\cap X}|= 2n/3}}\exp\left(\frac{-(1.006)^n}{4}\right)\\
&\le 2^{2N}\exp\left(\frac{-(1.006)^n}{4}\right)\\
&= 2^{4n+2cn}\exp\left(\frac{-(1.006)^n}{4}\right)\to 0\text{ as }n\to\infty.
\end{align*}
Therefore, with high probability, $X_t=0$, i.e. for almost all $t$-cones we have that $|{\mathcal K_t\cap\mathcal T}|>(1.006)^n$, which finishes the proof of the claim.
\end{proof}
    
This shows that, with high probability, our family $\mathcal T$ has non-empty intersection with every $t$-cone $\mathcal K_t$, therefore satisfying the second condition of Lemma~\ref{lem:pivot_points}.

Recall that an $s$-cone $\mathcal K_s$ is said to be bad if for all $S\in\mathcal K_s$ there exists some $T\in\mathcal T$ such that $S\subseteq T$.

We are now going to estimate the probability that any given $s$-cone is bad.
\begin{claim}\label{clm_2}
For any cone $\mathcal K_s$, $\mathbb{P}[\mathcal K_s\text{ is bad}]\leq (0.9997)^{n(1.0002)^n}$.
\end{claim}
\begin{proof}
Let $K_s=K_s(P,X)$ be an $s$-cone. In particular, this means that $|P|=\frac{n}{3}$, $|X|=n$, and $P\cap X=\emptyset$.

Instead of looking at the entire cone, we will evaluate $\mathbb{P}[\mathcal K_s'\text{ is bad}]$ for a suitable subfamily $\mathcal K_s'\subseteq \mathcal K_s$. The construction of the family $\mathcal K_s'$ will be in such a way that the events $\exists_{T\in\mathcal T}[S\subseteq T]$ are pairwise independent for all $S\in\mathcal K_s'$. A clean way to ensure this is to say that no two elements of $\mathcal K'_s$ `see' the same set in the $t$-level of $Q_N$.

Therefore, we define \textit{the $t$-neighbourhood of $S$}, $N_t(S)=\{T\subseteq [N] : |{T}|=t,\, S\subseteq T\}$. We now aim to pick $\mathcal K_s'$ in such a way that the neighbourhoods $N_t(S)$ are pairwise disjoint for $S\in\mathcal K_s'$.
        
Our construction is iterative. We start with some $S_1\in\mathcal K_s$. Then, we delete all the sets in $\mathcal K_s$ whose $t$-neighbourhood overlaps with the $t$-neighbourhood of $S_1$, and then repeat.
        
We now estimate how large the final family, $\mathcal K'_s$, is. The $t$-neighbourhoods of two sets overlap if and only if their union has size at most $t$. Let $A,B\in\mathcal K_s$ such that $|A\cup B|\geq t$. Then, since $|A\cup B|=|A|+|B|-|A\cap B|=2s-|A\cap B|$, we get that $2s-t=\frac{2n}{3}-\frac{cn}{2}-hn\geq|{A\cap B}|$.
        
We first compute how many sets of $K_s$ we have removed in step 1, in other words, how many sets there are in $\mathcal K_s$ such that their intersection with $S_1$ has size at least $\frac{2n}{3}-\frac{cn}{2}-hn$. Recall that $\mathcal K_s=\mathcal K_s(P,X)$, thus any set in $\mathcal K_s$ is comprised of $P$, union a subset of $X$ of size $\frac{n}{3}$. Therefore we count how many subsets of size $\frac{n}{3}$ of $X$ have intersection of size at least $\frac{n}{3}-\frac{cn}{2}-hn$ with $S_1\setminus P$. These will be all the subsets of $X$ we get by removing $i\leq\frac{cn}{2}+hn$ elements from $S_1\cap X=S_1\setminus P$, and adding  $i$ from the part of $X$ disjoint from $S_1$, which has size $\frac{2n}{3}$. Therefore, we remove exactly 
$$\sum_{i=0}^{cn/2+hn}{n/3\choose i}{2n/3\choose i}\le \left(\frac{cn}{2}+hn\right)\max\bigg\{{n/3\choose i}{2n/3\choose i}:0\leq i\leq cn/2+hn\bigg\}\text{ sets}.$$

To determine the maximum of the product ${n/3\choose i}{2n/3\choose i}$ for $0\leq i\leq cn/2+hn$, we will show that the quantity increases with $i$. To do that, we analyse the following ratio, for $i\geq 1$: $$\frac{{n/3\choose i}{2n/3\choose i}}{{n/3\choose i-1}{2n/3\choose i-1}}=\frac{(n/3+1-i)(2n/3+1-i)}{i^2}.$$
The fraction is at least 1 if and only if $\frac{2n^2}{9}+1+n\geq i(2+n)$. Therefore, if $i\leq\frac{2n}{9}$, the ratio is at least 1, hence the quantity is increasing, as claimed. Therefore, since $c/2+h=1/6+0.05<2/9$, we removed at most $$\left(\frac{cn}{2}+hn\right){n/3\choose {cn/2+hn}}{2n/3\choose{cn/2+hn}}\text{ sets.}$$

Next, by our inequalities for this type of binomial coefficients, given $\epsilon>0$, for large enough $n$ we have that $\left(\frac{cn}{2}+hn\right){n/3\choose {cn/2+hn}}{2n/3\choose{cn/2+hn}}$ is at most
$$(1/6+0.05)n\left(\frac{(1/3)^{1/3}(2/3)^{2/3}}{(1/6+0.05)^{2(1/6+0.05)}(1/6-0.05)^{1/6-0.05}(1/2-0.05)^{1/2-0.05}}+\epsilon\right)^n.$$
Since $\frac{(1/3)^{1/3}(2/3)^{2/3}}{(1/6+0.05)^{2(1/6+0.05)}(1/6-0.05)^{1/6-0.05}(1/2-0.05)^{1/2-0.05}}=1.88929\dots$, we get that for large $n$ we remove at most $(1.8893)^n$ sets.

Therefore, at each step of this process, we delete at most $(1.8893)^n$ elements of $\mathcal K_s$, which gives us that, for large enough $n$, we have
$$|{\mathcal K_s'}|\geq\frac{|{\mathcal K_s}|}{(1.8894)^n}\geq\frac{(1.8898)^n}{(1.8893)^n}\geq(1.0002)^n.$$
Finally, we upper bound the probability that the $s$-cone $\mathcal K_s$ is bad. We first have that, for $\epsilon>0$ small enough, and for $n$ large enough, the size of a $t$-neighbourhood of any $S\in\mathcal K_s$, $|N_t(S)|$, is
$${N-s\choose t-s}={\frac{4n}{3}+cn\choose \frac{cn}{2}+hn}\leq \left(\frac{(4/3+1/3)^{4/3+1/3}}{(1/6+0.05)^{1/6+0.05}(4/3+1/6-0.05)^{4/3+1/6-0.05}}+\epsilon\right)^n\leq (1.9041)^n.$$

Recall that, by constructions, the sets $N_t(S)$ for $S\in\mathcal K_s'$ are disjoint. We therefore get that
$$\mathbb{P}[\mathcal K_s\text{ is bad}]\leq \mathbb{P}[\mathcal K_s'\text{ is bad}]=\prod_{S\in\mathcal K_s'}\mathbb{P}[N_t(S)\cap\mathcal T\neq\emptyset].$$
After a union-bound we get that, for a given $S\in\mathcal K'_s$, $$\mathbb{P}[N_t(S)\cap\mathcal T\neq\emptyset]\leq\sum_{T\in N_t(S)}\mathbb{P}[T\in \mathcal T]=p|N_t(S)|.$$

Putting everything together we have $$\mathbb P[\mathcal K_s\text{ is bad}]\leq\prod_{S\in\mathcal K'_s}|N_t(S)|p\leq\prod_{S\in\mathcal K'_s}(1.9041)^n(0.525)^n\leq(0.9997)^{n|\mathcal K'_s|}.$$
Together with the fact that $|\mathcal K'_s|\geq (1.0002)^n$, we have that $\mathbb P[\mathcal K_s\text{ is bad}]\leq(0.9997)^{n(1.0002)^n}$, which finishes the claim.
\end{proof}
\begin{claim}\label{clm_3} With high probability there are no bad cones. 
\end{claim}
\begin{proof}
Let $X_\text{bad}$ be the random variable counting the number of bad $s$-cones, $\mathcal K_s$. By Claim~\ref{clm_2}, we can upper bound its expectation as follows:
\begin{align*}
\mathbb{E}[X_\text{bad}]=\sum_{\mathcal K_s}\mathbb{P}[\mathcal K_s\text{ is bad}]&\leq \sum_{\substack{X\subseteq[N]\\|{X}|=n}}\sum_{\substack{P\subseteq[N]\\|{P}|=n/3,\\P\cap X=\emptyset}}\mathbb{P}[\mathcal K_s(P,X)\text{ is bad}]\\&\leq 2^{2N}(0.9997)^{n(1.0002)^n}\leq 2^{4n+2cn}(0.9997)^{n(1.0002)^n}.
\end{align*}
Finally, by Markov's inequality, we have that $\mathbb{P}[X_\text{bad}\geq1]\leq\mathbb E[X]\leq 2^{4n+2cn}(0.9997)^{n(1.0002)^n}\to 0$ as $n\to\infty$. In other words, with high probability, $X_\text{bad}=0$, as claimed.
\end{proof}
Finally, we have that by Claim~\ref{clm_1} and Claim~\ref{clm_3}, with high probability, the family $\mathcal T$ does not generate any bad $s$-cones $\mathcal K_s$, and for every $t$-cone $\mathcal K_t$ the intersection $\mathcal T\cap\mathcal K_t$ is non-empty. Since $n$ is taken to be large enough, this implies that there must exist a family $\mathcal T$ with the above properties. Let that be $\mathcal T_1$.
    
Next, by construction, for every cone $s$-cone $\mathcal K_s$, there exists some $S\in\mathcal K_s$ such that $N_t(S)\cap\mathcal T_1=\emptyset$. Let $\mathcal S_1$ be the collection of all such $S$ for all $s$-cones $\mathcal K_s$. Thus, for all $S\in\mathcal S_1$ and $T\in\mathcal T_1$, $S\not\subseteq T$, and\\
$\bullet$ for all $\mathcal K_s$ there exists some $S\in\mathcal K_s\cap\mathcal S_1$, corresponding to the first part of Lemma~\ref{lem:pivot_points},\\
$\bullet$ for all $\mathcal K_t$ there exists some $T\in\mathcal K_t\cap\mathcal T_1$, corresponding to the second part of Lemma~\ref{lem:pivot_points},\\which finishes the proof.
\end{proof}
We now move on to Lemma~\ref{lem:above_pivots}, which follows from Lemma~\ref{lem:pivot_points} by taking complements. For completeness, we restate it below, and check all the conditions.
\setcounter{theorem}{4}
\begin{lemma}
For large enough $n$, there exist sets $\mathcal S_2\subseteq[N]^{(N-t)}$ and $\mathcal T_2\subseteq [N]^{(N-s)}$ with no $S\in\mathcal S_2$ and $T\in\mathcal T_2$ such that $S\subseteq T$, and the following properties:
\begin{enumerate}
\item For any $P,X\subseteq[N]$ with $|P|=\frac{N}{2}=n+\frac{cn}{2}$ and $|X|=n$, and $|P\cap X|= \frac{n}{3}$, there exists $S\in \mathcal S_2$ such that:
\begin{enumerate}
\item $P\cap X\subseteq S\cap X$
\item $S\setminus X\subseteq P\setminus X$
\item $|S\cap X|\leq \frac{2n}{3}$.
\end{enumerate}
\item For any $P,X\subseteq[N]$ with $|P|=N-\frac{n}{3}$, $|X|=n$, and $X\subseteq P$, there exists $T\in\mathcal T_2$ such that:
\begin{enumerate}
\item $T\cap X\subseteq P\cap X$
\item $P\setminus X\subseteq T\setminus X$
\item $|T\cap X|\geq \frac{2n}{3}$.
\end{enumerate}
\end{enumerate}
\end{lemma}
\begin{proof}[Proof of Lemma~\ref{lem:above_pivots}]
Let $\mathcal S_1\subseteq[N]^{(s)}$ and $\mathcal T_1\subseteq [N]^{(t)}$ be the sets that satisfy Lemma~\ref{lem:pivot_points}. We will show that $\mathcal S_2=\{[N]\setminus T : T\in\mathcal T_1\}$ and $\mathcal T_2=\{[N]\setminus S : S\in\mathcal S_1\}$ satisfy the conditions of Lemma~\ref{lem:above_pivots}. We need to check three properties.

First, there exist no $S\in\mathcal S_2,T\in\mathcal T_2$ such that $S\subseteq T$. Indeed, if $S\subseteq T$ for some $S\in\mathcal S_2, T\in\mathcal T_2$, then $[N]\setminus T\subseteq [N]\setminus S$. However, $[N]\setminus T\in\mathcal S_1$ and $[N]\setminus S\in \mathcal T_1$, a contradiction.

Next, we show that given $P,X\subseteq[N]$ such that $|{P}|=\frac{N}{2}=n+\frac{cn}{2}$, $|{X}|=n$, and $|{P\cap X}|= \frac{n}{3}$, there exists $S\in \mathcal S_2$ such that $P\cap X\subseteq S\cap X$, $S\setminus X\subseteq P\setminus X$, and $|{S\cap X}|\leq \frac{2n}{3}$.

Let $Q=[N]\setminus P$, thus $|{Q}|=\frac{N}{2}$. Moreover, since $Q\cap X=X\setminus P$ and $|P\cap X|=\frac{n}{3}$, we also have that $|Q\cap X|=\frac{2n}{3}$. Therefore, by the second part of Lemma~\ref{lem:pivot_points}, there exists $T\in\mathcal T_1$ such that $T\cap X\subseteq Q\cap X$, $Q\setminus X\subseteq T\setminus X$, and $|{T\cap X}|\geq \frac{n}{3}$.

Set $S=[N]\setminus T\in\mathcal S_2$. First, we have that $P\cap X=X\setminus Q\subseteq X\setminus T=S\cap X$.

Next, $S\setminus X=([N]\setminus T)\setminus X=([N]\setminus X)\setminus (T\setminus X)\subseteq ([N]\setminus X)\setminus (Q\setminus X)=([N]\setminus Q)\setminus X=P\setminus X$.

Finally, since $|T\cap X|\geq\frac{n}{3}$, we clearly have that $|S\cap X|\leq\frac{2n}{3}$. Thus, all conditions of the first part of Lemma~\ref{lem:above_pivots} are satisfied.

For the second part, we proceed similarly. Let $P,X\subseteq[N]$ such that $|{P}|=N-\frac{n}{3}$, $|{X}|=n$ and $X\subseteq P$. Then we show there  exists $T\in\mathcal T_2$ such that $P\setminus X\subseteq T\setminus X$, and $|{T\cap X}|\geq \frac{2n}{3}$.

Let $Q=[N]\setminus P$, thus $|{Q}|=\frac{n}{3}$. Moreover, since $X\subseteq P$, we also have that $Q\cap X=\emptyset$. By the first part of Lemma~\ref{lem:pivot_points}, there exists $S\in\mathcal{S}_1$ such that $S\setminus X\subseteq Q\setminus X$, and $|{S\cap X}|\leq \frac{n}{3}$.

Set $T=[N]\setminus S\in \mathcal T_2$. Since $|S\cap X|\leq\frac{n}{3}$, we of course have that $|T\cap X|\geq\frac{2n}{3}$. 

Lastly, $P\setminus X=([N]\setminus Q)\setminus X=([N]\setminus X)\setminus(Q\setminus X)\subseteq([N]\setminus X)\setminus(S\setminus X)=([N]\setminus S)\setminus X=T\setminus X$.

Thus the conditions of the second part of Lemma~\ref{lem:above_pivots} are also satisfied, which finishes the proof.
\end{proof}

\section{Proof of the main result}
In this section we show that $R(Q_n, Q_n)\geq2.7n+k$ for some constant $k$. All ideas of the proof have already been presented in great detail in the proof of Theorem~\ref{thm:qnqn_lower} -- we will increase the number of layers of the initial colouring, and for each pair of consecutive layers, we  modify the colouring with pivot sets, just as above. As such, we present the outline of the proof, as well as the Python code used to find the various numerical values.

We first analyse the above proof and the key facts that made all computations go through. 

The construction starts with a 6-layered colouring. Recall that this means sets of same size receive the same colour, and the number of colour transitions, while going upwards (or downwards) in the hypercube, plus one, is the number of layers of the colouring. If one wishes to keep the same starting point (the first and last layer), the only other quantities that could be varied are $c$ and $h$. Therefore, we view the proof as an optimisation problem: we have non-negative variables $c,h$ and we wish to maximise $c$, while ensuring all steps of the proof remain valid. 

It is clear that, as long as Lemma~\ref{lem:pivot_points} remains true for our new $c$ and $h$, the rest of the proof follows in the exact same way. Going through the proof of Lemma~\ref{lem:pivot_points} carefully, we identify the exact constraints on $c$ and $h$ that must be preserved.

We first look where the choice of the probability $p=q^n$, where $q$ is a constant between 0 and 1, has played a role. We see that in Claim~\ref{clm_1} we needed $|{\mathcal K_t}|q>1$, and in Claim~\ref{clm_2} we needed $1>|{N_t}|q$, for all $t$-neighbourhoods. Therefore, such choice of $p$ is possible if and only if $|{\mathcal K_t}|>|{N_t}|$. Using the bound we derived in the proofs of these claims, we can translate this into a condition $c$ and $h$ must satisfy, namely
$${2n/3\choose n/3}{2n/3 + cn/2\choose hn}>{\frac{4n}{3}+cn\choose \frac{cn}{2}+hn}.$$
We refer to this as the \textit{probability constraint}.

The other, non-probabilistic constraint appears in Claim~\ref{clm_2}, during the construction of $\mathcal K_s'$. Recall that at each step we discarded the sets that had large intersection with at least one of the sets already chosen. We require that the number of elements we remove is exponentially smaller than $|{\mathcal K_s}|\geq {n\choose n/3}$. We established that, as long as $c/2+h<\frac{2}{9}$, the number of elements we remove is at most $$\bigg(\frac{cn}{2}+hn\bigg){n/3\choose{cn/2+hn}}{2n/3\choose{cn/2+hn}}.$$ A quick check shows that, for fixed $i\leq\frac{2}{9}$, $(in){n/3\choose{in}}{2n/3\choose{in}}$ is exponentially smaller than ${n\choose n/3}$. Therefore, the restriction $c/2+h<\frac{2}{9}$ guarantees the validity of this part of the proof. We refer to this as the \textit{intersection constraint}.

The intersection constraint gives a direct upper bound for this method, namely $c<\frac{4}{9}$. It also means that increasing $h$ comes at the cost of $c$. On the other hand, the probability constraint is where we can use $h$ to increase $c$. This is because increasing $h$ slightly above zero increases the left hand side more than the right hand side, thus allowing for larger $c$, while still not violating the probability constraint. Nevertheless, this approach is somehow optimised, and cannot result in an upper bound of more than $(2+\frac{4}{9})n$, as already explained.

However, another approach is to increase the number of layers of the colouring we are starting with, and then we get a much better bound. Below we explain in detail the general strategy and limitations, as well as the following verified result, for which the parameters have been found using the code presented in the Appendix.
\setcounter{theorem}{0}
\begin{theorem}\label{bestbound}
There exists a constant $k$ such that $R(Q_n,Q_n)>2.7n+k$.
\end{theorem}
\setcounter{theorem}{5}
\begin{proof}[Outline of the proof.]
Let $N=(2+c)n$, $n$ big enough and also such that all the quantities involving a multiple of $n$ are positive integers. The proof starts with a layered colouring of $Q_N$, with $4L+2$ layers. The modification of this colouring is symmetric in the upper half of the hypercube, therefore we focus on the $2L+1$ layers below $N/2$. Let all sets of size less than $b_0n$ be blue, for some $b_0>0$ to be chosen later.

We now group the rest of the layers into $L$ consecutive pairs. For each pair of layers we have an equivalent of Lemma~\ref{lem:pivot_points} that will allow us to modify the colouring in these two layers as before.

Every pair of layers will have the following parameters: a value of $c_i$ indicating how many levels of the hypercube we are forced to skip in that region, and an $h_i$, $r_i$ and $b_i$, the equivalent of $h$, $n/3$ and $n/3$ in the proof of Lemma~\ref{lem:pivot_points}.

More precisely, let $b_0n=l_1n<l_2n<\dots<l_{L+1}n=N/2$ be the levels of the hypercube above the blue sets. For every $1\leq i\leq L$, we find two families of sets $\mathcal S_i$ and $\mathcal T_i$ with the same properties as in Lemma~\ref{lem:pivot_points}, except $s$ is replaced by $s_in=(l_i+r_i)n$, $t$ is replaced by $t_in=(l_i+r_i+c_i+h_i)n$. Moreover, in part 1, `$P\cap X=\emptyset$' is replaced by $|P\cap X|=\sum_{k=1}^{i-1}r_kn$, `$|P|=n/3$' is replaced by $|P|=l_in$, and `$|S\cap X|\leq n/3$' is replaced by $|S\cap X|\leq\sum_{k=1}^{i}r_kn$. Similarly, in part 2, `$|P|=N/2$' is replaced by $|P|=l_{i+1}n$, `$|P\cap X|=2n/3$' is replaced by $|P\cap X|=n-\sum_{k=1}^Lr_kn-\sum_{k=i+1}^Lb_kn$, and `$|T\cap X|\geq n/3$' is replaced by $|T\cap X|\ge n-\sum_{k=1}^Lr_kn-\sum_{k=i}^Lb_kn$. Moreover, $l_{i+1}-l_i=b_i+c_i+r_i$. This is illustrated below.
\begin{center}
\includegraphics[width=25em]{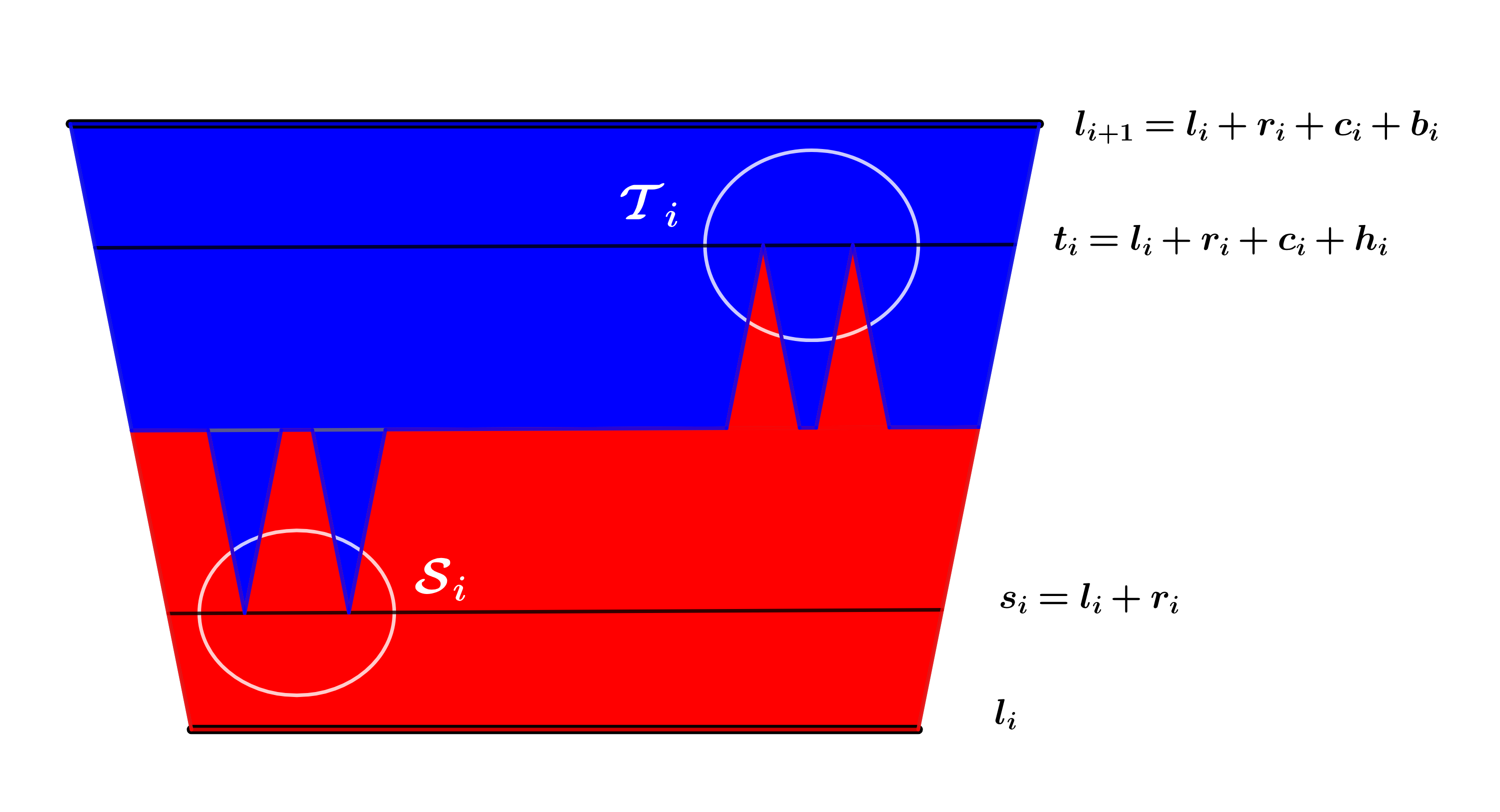}
\end{center}
Our total value for $c$ will then be $c=2\sum_{i=1}^L c_i$.

In each pair of layers, the parameters $r_i$ and $b_i$ tell us how much a monochromatic embedding is allowed to travel in that region. This is captured in Lemma~\ref{lem:pivot_points} by the intersection condition with $X$, the ground set of the embedding. Similarly here, for the $\mathcal S_i$ family, corresponding to working with a red embedding, we start with $P$ such that $|P|=l_in$ and $|P\cap X|=\sum_{k=1}^{i-1}r_kn$, and obtain $S\in\mathcal S_i$ such that $|S\cap X|\leq\sum_{k=1}^{i} r_kn$, thus we travel at most $r_in$ levels. Similarly, for the $\mathcal T_i$ family, corresponding to working with a blue embedding, we start with $P$ such that $|P|=l_{i+1}n$ and $|P\cap X|= n-\sum_{k=1}^Lr_kn-\sum_{k=i+1}^Lb_kn$, and obtain $T\in\mathcal T_i$ such that $|T\cap X|\ge n-\sum_{k=1}^Lr_kn-\sum_{k=i}^Lb_kn$, so again we have a travel of at most $b_in$ levels. The other travel parameter is the starting level, either from the top or the bottom, namely $b_0n$. Noting that a blue travel in a layer becomes a red travel in the complement layer (above $N/2$), we must ensure that the embedding cannot travel all the way to $n$, thus we need
$$b_0+\sum_{i=1}^L r_i+\sum_{i=0}^Lb_i\le 1.$$
The same way as in the proof of Theorem~\ref{thm:qnqn_lower}, we satisfy this condition by simply taking $b_0=r_i=b_i=\frac{1}{2L+1}$ for all $1\leq i\leq n$. Note that this is consistent with $l_{L+1}n=\sum_{i=0}^Lb_{i}n+\sum_{i=1}^Lc_in+\sum_{i=1}^Lr_in=(1+\frac{c}{2})n$.

We now move on to the intersection and probability constraints, and how they change with more layers, particularly in this setup. By design, we will have an instance of each of these constraints for each pair of layers.

For the \textbf{intersection constraint}, we initially had that ${n/3\choose (c/2+h)n}{2n/3\choose (c/2+h)n}$ grows slower than $|{\mathcal K_s}|\ge{n\choose n/3}$. That was to ensure that, whilst making the $t$-neighbourhoods disjoint, we do not remove too many elements for the $s$-cone. In this setup, between $l_in$ and $l_{i+1}n$ an $s_in$-cone will have the form $\{P\cup P':|P'|=r_in, P'\subseteq X\setminus P\}$. Based on this we get $|\mathcal K_{s_in}|\geq{(1-\sum_{k=1}^{i-1}r_k)n\choose r_in}$.

We want to eliminate sets that have union at most $t_in$. Fix $S\in\mathcal K_{s_in}$. We want to delete all sets $S'\in\mathcal K_{s_in}$ such that $|S\cup S'|\leq t_in$, or equivalently $|S\cap S'|\geq l_in+r_in-(c_i+h_i)n$. As before, the elements of an $s_in$-cone are of the form: $P$ union a subset of $X\setminus P$ of size $r_in$. Thus, we count the number of subsets of $X\setminus P$ of size $r_in$ that have intersection of size at least $r_in-(c_i+h_i)n$ with another given subset of size $r_in$. Since $|X\setminus P|=n-(\sum_{k=1}^{i-1}r_k)n$, we remove at most $\sum_{k=1}^{(c_i+h_i)n}{r_in\choose k}{(1-\sum_{k=1}^ir_k)n\choose k}$ elements. Similar computations as before show that the maximal binomial coefficient is when $k=\bigg(r_i-\frac{r_i^2}{1-\sum_{k=1}^{i-1}r_k}\bigg)n.$ 

Therefore we need $c_i+h_i\leq r_i-\frac{r_i^2}{1-\sum_{k=1}^{i-1}r_k}$, and $${r_in\choose (c_i+h_i)n}{(1-\sum_{k=1}^ir_k)n\choose (c_i+h_i)n}<<{(1-\sum_{k=1}^{i-1}r_k)n\choose r_in}.$$

For the \textbf{probability constraint}, we once again want that the size of a $t_in$-cone, $|\mathcal K_{t_in}|$ is strictly greater than the size of a $t_in$-neighbourhood of set of size $s_in$, $|{N_{t_in}}|={N-s_in\choose (t_i-s_i)n}={N-s_in\choose c_in+h_in}$. Before we lower bound the size of a $t_in$-cone, we introduce the following notation: for every $1\leq i\leq L$ we define $B_i=\sum_{k=0}^{i-1}b_k$. Recall that we are working under the assumption that $\sum_{k=0}^Lb_k+\sum_{k=1}^L r_k=1$. 

Therefore, the conditions for the family $\mathcal T_i$ start with $P$ such that $|P|=l_{i+1}n$ and $|P\cap X|=B_{i+1}n$, and obtain $T\in\mathcal T_{i}$ such that $|T\cap X|\geq B_in$. As before, by looking at sets of the form $(P\setminus X)\cup S'\cup S''$ and of size $t_in$, where $S'\subseteq P\cap X$, $|S'|=B_in$, and $S''\subset [N]\setminus(P\cup X)$, we get that a $t_in$-cone has size at least ${B_{i+1}n\choose B_in}{N-l_{i+1}n-n+B_{i+1}n\choose h_in}$, hence our probability constraint is
$${B_{i+1}n\choose B_in}{N-l_{i+1}n-n+B_{i+1}n\choose h_in}>{N-s_in\choose c_in+h_in}.$$

The constraints derived above have been implemented in the code presented in the Appendix, which is a numerical optimiser for the value of $c$. It was also used to find the several numerical values present in the proofs of Theorem~\ref{thm:qnqn_lower} and Lemma~\ref{lem:pivot_points}. The largest value of $c$ we calculated is $0.7$ with $L=150$, which indeed gives that $R(Q_n,Q_n)>2.7n+k$ for some constant $k$, and sufficiently large $n$.
\end{proof}
We end this section with the observation that, under all the constraints presented above, one cannot hope to find $c\geq1$. Recall that  we have taken $r_i=\frac{1}{2L+1}$ for all $1\leq i\leq L$, and, from the intersection constraint, we require $c_i+h_i\leq r_i-\frac{r_i^2}{1-\sum_{k=1}^{i-1}r_k}$. Summing all these inequalities we get
$$\sum_{i=1}^L (c_i+h_i)\leq\sum_{i=1}^L \left(r_i-\frac{r_i^2}{1-\sum_{k=1}^{i-1}r_k}\right).$$
Since $c=2\sum_{i=1}^Lc_i$ and $r_i=\frac{1}{2L+1}$ for all $1\leq i\leq L$, we have that
$$\frac{c}{2}+\sum_{i=1}^Lh_i\leq\frac{L}{2L+1}-\sum_{i=1}^L\frac{\frac{1}{(2L+1)^2}}{1-\frac{i-1}{2L+1}}\leq\frac{L}{2L+1}-L\frac{\frac{1}{(2L+1)^2}}{1-0}=\frac{2L^2}{(2L+1)^2}.$$
Multiplying both sides by 2 we get that $c+2\sum_{i=1}^Lh_i\leq\left(\frac{2L}{2L+1}\right)^2$, and so $c<1$. Furthermore, as $L\to\infty$ the right hand side goes to 1, suggesting that the more layers we work with, the better the bound, with $c$ potentially achieving all values in the interval $[0,1)$. However, we need a better understanding of the probability constraints, and whether the sum of the $h_i$ variables can be made arbitrarily close to zero as $L$ increases, in order to prove this.

\section{Limitations and future directions}
As explained above, it is not entirely clear, although it is believable, that the sum $\sum_{i=1}^Lh_i$ can be legally made to go to 0, as the number of layers, thus $L$, goes to infinity.

Nevertheless, in the proof of Theorem~\ref{bestbound}, we made a convenient, yet potentially very restrictive choice. Namely, we took all $b_i$ and $r_i$ to be equal to $\frac{1}{L+1}$. It is therefore natural to ask whether a different choice could lead to a better bound, one that could even potentially give $c>1$. 

One clue that this avenue could be fruitful comes from noting that the intersection constraints bound $c_i$ in terms of $r_i$, and not in terms of $b_i$. Therefore, increasing the $r_i$ values might make enough room and allow $c$ to go above 1. However, since $\sum_{i=0}^Lb_i+\sum_{i=1}^Lr_i\leq1$, increasing the $r_i$ immediately implies decreasing the $b_i$, so there is a delicate balance to be kept at all times in order for the structure of the proof to hold. At this stage we believe that some numerical work should indicate whether this approach might indeed yield better bounds or whether the other constraints are still too strong to allow for $c\ge 1$. We do however strongly believe the following.
\begin{conjecture}\label{conj:3-e}
Let $\epsilon>0$. Then there exists $L$ and a constant $a$ such that the above type of colouring ($2L$ pairs of layers modified by 2 families of pivots) gives $$R(Q_n,Q_n)\geq(3-\epsilon)n+a.$$   
\end{conjecture}
One could ask if changing the structure of the colouring ever so slightly might be beneficial. In our proof we have shifted the $t_in$ levels up by an extra $h_in$ term, in contrast with the $s_in$ levels. What if we also shift the $s_in$ levels down by a similar term? Surprisingly, this makes all our bounds tighter, thus giving worse results for $c$.

Finally, how suitable is this type of colouring if we change the question to the off-diagonal version, $R(Q_m, Q_n)$? This immediately comes with a lot of broken symmetry, which we heavily relied on while lower bounding $R(Q_n,Q_n)$. For example, we would need a flexible, yet sturdy enough way of estimating binomials of the form ${An+B m\choose pn+qm}$. If one such general estimate is found, the envisioned lower bound will be of the form $(1+\frac{c}{2})(m+n)$.

A different way of keeping all the binomials in terms of only one of $m$ or $n$ is to change the colouring slightly. For example, in $Q_N$, colour the middle $n-m$ levels blue, and the rest of them using our `pivotal' colouring techniques. This approach is tailored, if successful, to give a bound of the form $R(Q_m,Q_n)\ge (1+c)m+n$. 
 
Our feeling is that breaking symmetries is here an inconvenience, and not a true hurdle. Nevertheless, as with the diagonal case, we expect the same limitation, namely that $c<1$.
\bibliographystyle{amsplain}
\bibliography{bib}

\providecommand{\bysame}{\leavevmode\hbox to3em{\hrulefill}\thinspace}
\providecommand{\MR}{\relax\ifhmode\unskip\space\fi MR }
\providecommand{\MRhref}[2]{%
  \href{http://www.ams.org/mathscinet-getitem?mr=#1}{#2}
}
\providecommand{\href}[2]{#2}
\begin{thebibliography}{1}

\bibitem{axenovichBooleanLatticesRamsey2017}
Maria Axenovich and Stefan Walzer, \emph{Boolean {{Lattices}}: {{Ramsey Properties}} and {{Embeddings}}}, Order \textbf{34} (2017-07-01), no.~2, 287--298.

\bibitem{log}
Maria Axenovich and Christian Winter, \emph{Diagonal poset {R}amsey numbers}, Discrete Mathematics \textbf{349} (2026).

\bibitem{constructible}
Tom Bohman and Fei Peng, \emph{A {C}onstruction for {B}oolean {C}ube {R}amsey {N}umbers}, Order \textbf{40} (2023), no.~2, 327--333.

\bibitem{cox}
Christopher Cox and Derrick Stolee, \emph{Ramsey {N}umbers for {P}artially-{O}rdered {S}ets}, Order \textbf{35} (2018), no.~3, 557--579.

\bibitem{off-diagonal}
D{\'a}niel Gr{\'o}sz, Abhishek Methuku, and Casey Tompkins, \emph{Ramsey {N}umbers of {B}oolean {L}attices}, Bulletin of the London Mathematical Society \textbf{55} (2023), no.~2, 914--932.

\bibitem{winter}
Christian Winter, \emph{Erd{\H{o}}s-{H}ajnal {P}roblems for {P}osets}, Order \textbf{42} (2025), 509--527.

\end{thebibliography}
\Addresses

\newpage
\appendix
\section*{Appendix -- Python code to optimise $c$ for large $L$}




In this appendix we explain the Python code used to optimise the parameters that give us Theorem~\ref{bestbound}. This code can also be found at \href{https://github.com/AdriWessels/PosetRamseyLowerBound2026}{this GitHub repository}. The parameters we optimise over are a $c_i$ value and a $h_i$ value for each layer. The code sets up the relevant constraints for these parameters as discussed in the proof outline of Theorem~\ref{bestbound}, and then calls a function from \href{https://scipy.org/}{SciPy} to maximise the total value of $c=2\sum_{i=1}^Lc_i$.

In our proof, $c_i$ and $h_i$ are all coefficients of $n$. Other relevant quantities in the proof can all be represented as some base value raised to the power of $n$, thus we only focus on the base value for these. Moreover, by the symmetry of the construction under complements, we only need to do the calculations for the bottom half of the hypercube.

The first part of the code is our setup. We have the required imports, and define a function which performs the following estimate:
$$\binom{Cn}{dn}\approx \left(\frac{C^C}{d^d(C-d)^{C-d}}\right)^n.$$
\lstset{
    language=Python,
    basicstyle=\ttfamily\small,
    keywordstyle=\color{blue},
    stringstyle=\color{red},
    commentstyle=\color{green!60!black},
    numbers=left,
    showstringspaces=false,
    breaklines=true,           
    breakatwhitespace=true, 
    tabsize=4,
    backgroundcolor=\color{gray!10},
    frame=single,                          
    rulecolor=\color{lightgray},           
    framesep=5pt,                          
    framexleftmargin=1.7em
}

\begin{lstlisting}
from scipy.optimize import minimize
from math import log
import numpy as np
from collections import namedtuple

# Small epsilon to approximate strict inequality
epsilon = 1e-6

# Log everything to avoid floating point precision issues
def choose_base_log(C,d):
	if d <= 0 or d >= C:
		return 0
	return (C * log(C)) - (d * log(d)) - ((C-d) * log(C-d))
\end{lstlisting}

For each layer, we need to know the values corresponding to other layers to be able to calculate the relevant quantities. This includes the bottom level, the top level, and the equivalents of levels $s_in$ and $t_in$. We predefine these values, excluding the contributions from the $c_i$ and $h_i$ parameters.

There is also other data we can predefine. Recall that the proof works by only allowing an embedding to gain at most a certain number of levels in each layer, so that it runs out levels at the end. We must therefore track how many levels we have allowed a red embedding to gain in the lower layers, how many levels we allowed in the upper layers, and how many it is allowed to gain in this layer. In principle, these are also parameters which can be optimised over, but we fixed these to have a total of $n$ levels, divided equally over all layers.

\begin{lstlisting}[firstnumber=15]
# bottom = bottom level of the layer as frac of n, doesn't include contributions from c
# top = top level of layer as frac of n, doesn't include contributions from c
# red_level = what frac of n have we already climbed in previous layers for red at the bottom
# red_climb = what frac of n are we allowed to climb in the red part for this layer
# blue_level = what frac of n have we already climbed in previous layers for blue at the top
# blue_climb = what frac of n are we allowed to climb in the blue part for this layer
Layer = namedtuple("Layer", ["index", "bottom", "top", "red_level", "red_climb", "blue_level", "blue_climb", "num_layers"])

def make_layers(num_layers):
	denom = num_layers * 2 + 1
	layers = []
	for i in range(num_layers):		# loops from bottom of poset to middle
		layers.append(Layer(
			i,				        # index
			(1+2*i)/denom,			# bottom
			(3+2*i)/denom,			# top
			i/denom,				# red_level
			1/denom,				# red_climb
			(2 + i) / denom,		# blue_level
			1/denom,				# blue_climb
			num_layers
		))
	return layers

# top = bottom + red_climb + blue_climb
#           = red_level + blue_level + red_climb
# bottom = red_level + blue_level - blue_climb
\end{lstlisting}

Using the predefined constants and our parameters, we define how to calculate various quantities used in the proof.

\begin{lstlisting}[firstnumber=43]
# Define N
def var_N(vars, layer: Layer):
	c = vars[:layer.num_layers]
	return 2 + 2 * sum(c)

# Define s
def var_s(vars, layer: Layer):
	c = vars[:layer.num_layers]
	return layer.bottom + sum(c[:layer.index]) + layer.red_climb

# Define t
def var_t(vars, layer: Layer):
	c = vars[:layer.num_layers]
	h = vars[layer.num_layers:]
	return var_s(vars, layer) + c[layer.index] + h[layer.index]

# Define the top level of a layer including contributions from c
def var_top(vars, layer: Layer):
	c = vars[:layer.num_layers]
	return layer.top + sum(c[:layer.index + 1])

# Calculate K_s
def Ks_log(vars, layer: Layer):
	return choose_base_log(1 - layer.red_level, layer.red_climb)

# Calculate K_t
def Kt_log(vars, layer: Layer):
	c = vars[:layer.num_layers]
	h = vars[layer.num_layers:]
	return choose_base_log(layer.blue_level, layer.blue_climb) + choose_base_log(var_N(vars, layer) - layer.top - sum(c[:layer.index+1]) + layer.blue_level - 1, h[layer.index])

# Calculate t-neighbourhood of an s-set
def Nt_log(vars, layer: Layer):
	return choose_base_log(var_N(vars, layer)-var_s(vars, layer),var_t(vars, layer)-var_s(vars, layer))

# Calculate how many sets in K_s have union of size at most t with a specified set.
def Nsect_log(vars, layer: Layer):
	c = vars[:layer.num_layers]
	h = vars[layer.num_layers:]
	val = h[layer.index] + c[layer.index]
	return choose_base_log(layer.red_climb, val) + choose_base_log(1 - layer.red_level - layer.red_climb, val)
\end{lstlisting}

We want to define the constraints we require these quantities to stay within. These correspond to those discussed in the proof outline of Theorem~\ref{bestbound}: the intersection constraint and the probability constraint.

\begin{lstlisting}[firstnumber=85]
# We need level t to be inside the layer, so t <= top -> top - t >= 0
# This constraint should be covered by the Nsect_max_constraint, so we don't include it in the optimisation, but we can check it afterwards
def t_lessThan_Top_constraint(vars, layer: Layer):
	c = vars[:layer.num_layers]
	return layer.top + sum(c[:layer.index + 1]) - var_t(vars, layer) - epsilon

# I-Constraint: This comes from finding a large subset of K_s(P,X) with on large intersections
# h+c <= RC - RC^2/(1-RL) -> RC - RC^2/(1-RL) - (h+c) >= 0
def Nsect_max_constraint(vars, layer: Layer):
	c = vars[:layer.num_layers]
	h = vars[layer.num_layers:]
	return layer.red_climb - (layer.red_climb ** 2) / (1 - layer.red_level) - (h[layer.index] + c[layer.index]) - epsilon

# p-constraint: This combines the constraints that 1 <= pt * Kt and pt * Nt <= 1, which is equivalent to Kt/Nt >= 1 -> Kt - Nt >= 0
def room_for_pt_constraint(vars, layer: Layer):
	return Kt_log(vars, layer) - Nt_log(vars, layer) - epsilon

# Constraint: 1 <= K' -> K' - 1 >= 0
# This constraint should be covered by the Nsect_max_constraint, so we don't include it in the optimisation, but we can check it afterwards
def K_bad_subfamily_constraint(vars, layer: Layer):
	return (Ks_log(vars, layer) - Nsect_log(vars, layer)) - epsilon

# Constraint: N - top - n + blue_level >= ht -> N - top - n + blue_level - ht >= 0
# We need this to lower bound K_t in Case A, otherwise we cannot take the extra hn elements
def room_for_h_constraint(vars, layer: Layer):
	h = vars[layer.num_layers:]
	return var_N(vars, layer) - var_top(vars, layer) - 1 + layer.blue_level - h[layer.index] - epsilon

# Some extra constraints which I conjecture are good based on optimal values seen before
# We do not need to check that these hold, but hopefully they help the optimiser

# c_i should be increasing
def c_increasing_constraint(vars, i, num_layers):
	c = vars[:num_layers]
	return c[i+1] - c[i]

# h_i should be decreasing
def h_decreasing_constraint(vars, i, num_layers):
	h = vars[num_layers:]
	return h[i] - h[i+1]
\end{lstlisting}

Finally, we setup our objective function (to maximise $\sum_i c_i$) and run the optimiser in a loop while ensuring all constraints hold.

\begin{lstlisting}[firstnumber=126]
# Objective function: maximise c (equivalent to minimising -c)
def objective(vars, num_layers):
	c = vars[:num_layers]
	return -sum(c)  # Negate since we minimise in scipy

def optimiseForNLayers(num_layers):
	layers = make_layers(num_layers)

	# Initial guess
	starting_point = [0.001] * num_layers   # c
	starting_point += [0.0001] * num_layers # h

	# Bounds
	bounds = [(0, 2)] * num_layers  # c
	bounds += [(0, 1)] * num_layers # h

	# Constraints
	constraints = [{'type': 'ineq', 'fun': room_for_pt_constraint, 'args': (layer,)} for layer in layers]
	constraints += [{'type': 'ineq', 'fun': Nsect_max_constraint, 'args': (layer,)} for layer in layers]
	constraints += [{'type': 'ineq', 'fun': room_for_h_constraint, 'args': (layer,)} for layer in layers]
	constraints += [{'type': 'ineq', 'fun': c_increasing_constraint, 'args': (i,num_layers,)} for i in range(num_layers - 1)]
	constraints += [{'type': 'ineq', 'fun': h_decreasing_constraint, 'args': (i,num_layers,)} for i in range(num_layers - 1)]

	guess = starting_point
	best = 0
	best_vars = None

	for _ in range(3): # Run multiple times in case it fails some of the constraints
		result = minimize(
			objective, 
			guess,
			args=(num_layers,),
			bounds=bounds,
			constraints=constraints,
			method="trust-constr", 
			hess=lambda x, _: np.zeros((len(x), len(x))),
			tol=1e-9
		)

		# Negate to get maximum c, double for reflection across middle layer using complements
		max_c = -result.fun * 2
		optimal_vars = result.x	

		# Make sure the values we got from the optimiser actually satisfy the constraints
		test = True
		for i in range(num_layers):
			layer = layers[i]
			
			good = (room_for_pt_constraint(optimal_vars, layer) > 0)
			test = test and good
			if not good:
				print("Layer ", i, ": ", "room_for_pt_constraint: ", room_for_pt_constraint(optimal_vars, layer))
			
			good = (K_bad_subfamily_constraint(optimal_vars, layer) > -epsilon)
			test = test and good
			if not good:
				print("Layer ", i, ": ", "K_bad_subfamily: ", K_bad_subfamily_constraint(optimal_vars, layer))
			
			good = (Nsect_max_constraint(optimal_vars, layer) > 0)
			test = test and good
			if not good:
				print("Layer ", i, ": ", "Nsect_max: ", Nsect_max_constraint(optimal_vars, layer))
			
			good = (t_lessThan_Top_constraint(optimal_vars, layer) > -epsilon)
			test = test and good
			if not good:
				print("Layer ", i, ": ", "t_lessThan_Top_constraint: ", t_lessThan_Top_constraint(optimal_vars, layer))

            good = (room_for_h_constraint(optimal_vars, layer) > 0)
			test = test and good
			if not good and debug:
				print("Layer ", i, ": ", "room_for_h_constraint: ", room_for_h_constraint(optimal_vars, layer))

		if test:
			print("All constraints good")
			print(max_c)
			print(optimal_vars)
			if max_c > best:
				best = max_c
				best_vars = optimal_vars
		else:
			print("Some constraints off")
		
		guess = optimal_vars # Start the next run of the optimiser from the previous result
	
	return best, best_vars

best, best_vars = optimiseForNLayers(150)

print("#############################################")
print("Best: ", best)
print("Best parameters: ", best_vars)
\end{lstlisting}

Running this code with $L=150$ gives $c>0.7$, which is exactly Theorem~\ref{bestbound}. This corresponds to 602 total layers (300 pairs of layers).

\end{document}